\documentclass[11pt]{amsart}
\usepackage{longtable}
\usepackage{array}
\usepackage{multicol}
\usepackage{graphicx}
\usepackage[colorlinks=true,citecolor=black,linkcolor=black,urlcolor=blue]{hyperref}

\makeatletter
 \def\@textbottom{\vskip \z@ \@plus 1pt}
 \let\@texttop\relax
\makeatother
\usepackage{amsmath, amssymb, amsbsy, amsfonts, amsthm, latexsym, amsopn, amstext, amsxtra, euscript, amscd, color, mathrsfs}
\usepackage[normalem]{ulem}
\usepackage{soul}

\usepackage{cite}

\makeatletter

\@namedef{subjclassname@2020}{%
  \textup{2020} Mathematics Subject Classification}
\makeatother

\PassOptionsToPackage{hyphens}{url}\usepackage{hyperref}
 
 \usepackage[capbesideposition=outside,capbesidesep=quad]{floatrow}

\restylefloat{table}
\restylefloat{table}
         
\usepackage{multirow,caption}
            
\usepackage{amscd}
\usepackage{color,enumerate}

\newcommand{\RNum}[1]{\lowercase\expandafter{\romannumeral #1\relax}}

\usepackage[colorinlistoftodos,prependcaption,textsize=tiny]{todonotes}

\theoremstyle{plain}
\newtheorem{thm}{Theorem}[section]
\newtheorem{lem}[thm]{Lemma}

\newtheorem{prop}[thm]{Proposition}

\newtheorem{rmk}[thm]{Remark}

\newtheorem{thm-con}[thm]{Theorem-Conjecture}
\numberwithin{equation}{section}
\newtheorem{def1}[thm]{Definition}

\def\F{{\mathbb F}}
\def\Tr{{\rm Tr}}

\begin{document}

\title[Some new classes of permutation polynomials]{Some new classes of permutation polynomials and their compositional inverses}

 \author[S. U. Hasan]{Sartaj Ul Hasan}
 \address{Department of Mathematics, Indian Institute of Technology Jammu, Jammu 181221, India}
  \email{sartaj.hasan@iitjammu.ac.in}
  
   \author[R. Kaur]{Ramandeep Kaur}
  \address{Department of Mathematics, Indian Institute of Technology Jammu, Jammu 181221, India}
  \email{2022rma0027@iitjammu.ac.in}

 \thanks{The first named author is partially supported by Core Research Grant CRG/2022/005418 from the Science and Engineering Research Board, Government of India. The second named author is supported by the Prime Minister’s Research Fellowship PMRF ID-3003658 at IIT Jammu.}



\begin{abstract}
We focus on the permutation polynomials of the form $L(X)+\Tr_{m}^{3m}(X)^{s}$ over $\F_{q^3}$, where $\F_q$ is the finite field with $q=p^m$ elements, $p$ is a prime number, $m$ is a positive integer, $\Tr_{m}^{3m}$ is the relative trace function from $\F_{p^{3m}}$ to $\F_{p^{m}}$, $L(X)$ is a linearized polynomial over $\F_{q^{3}}$, and $s>1$ is a positive integer. More precisely, we present six new classes of permutation polynomials over $\F_{q^3}$ of the aforementioned form: one class over finite fields of even characteristic, three classes over finite fields of odd characteristic, and the remaining two over finite fields of arbitrary characteristic. Furthermore, we show that these classes of permutation polynomials are inequivalent to the known ones of the same form. We also provide the explicit expressions for the compositional inverses of each of these classes of permutation polynomials.
\end{abstract}
\maketitle

\section{Introduction}
For a prime power $q=p^m$, let $\F_q$ denote the finite field with $q$ elements, $\F_q^*$  the multiplicative cyclic group of all non-zero elements of $\F_q$, and $\F_q[X]$ the ring of polynomials over $\F_q$ in the indeterminate $X$. A polynomial $f(X)\in \F_q[X]$ is a permutation polynomial (PP) if the induced map $c \mapsto f(c)$ is a bijection from $\F_q$ to itself. It may be noted that permutation polynomials were first studied by Hermite~\cite{CC} over the finite fields of prime order. Later, Dickson~\cite{LE} investigated PPs over arbitrary finite fields. For any PP $f(X)\in \F_q[X]$, there exists a unique  polynomial $f^{-1}(X)\in \F_q[X]$ modulo $X^q-X$ such that $f^{-1}(f(X))\equiv X \pmod{X^q-X}$. The polynomial  $f^{-1}$ is referred to as the compositional inverse of $f$ over $\F_q$. Finding PPs and their compositional inverses holds significant importance from both mathematical and practical perspectives. In fact, PPs and their compositional inverses have a wide range of applications in cryptography\cite{RSA,Schwenk_Cr_98}, coding theory \cite{DH,Chapuy_C_07}, and combinatorial design theory ~\cite{DY}. In 1991, Mullen~\cite{M} posed a problem concerning the computation of the coefficients of the compositional inverses of permutation polynomials. The significance of explicitly determining the compositional inverses of PPs stems from the fact that PPs and their inverses have plethora of applications. For example, in a block cipher with a Substitution-Permutation Network (SPN) structure, the compositional inverse of the substitution box (S-box) is crucial throughout the decryption process. Generally, constructing PPs and determining their compositional inverses is a challenging problem that has garnered the attention of many researchers.

We now provide a brief account of the development and motivation behind the specific form of permutation polynomials that we are considering in this paper. In 2003,  Helleseth and Zinoviev~\cite{HZ} studied PPs of the form $\left(\frac{1}{X^2+X+\delta}\right)^{2^{\ell}}+X$ to establish new Kloosterman sums identities over $\F_{2^{m}}$, where $\ell\in \{0,1\}$ and $\delta\in \F_{2^m}$. Inspired by their work, Yuan and Ding ~\cite{YD,YD1} initiated an investigation for PPs of the form 
 $f_{1}(X):=(X^{p^i}-X+\delta)^s+L(X)$  over $\F_q$, where $i,s$ are positive integers, $\delta\in \F_q$ and $ L(X):=\sum_{i=0}^{m-1}\alpha_{i}X^{p^i}$ is a linearized polynomial in $\F_q[X]$. Following this, several classes of PPs of the form $f_{1}(X)$ have been constructed in the literature~\cite{LHT, ZZH, ZH1, ZH, ZYP}. The study of permutation polynomials of the shape $f_{1}(X)$ may be divided into two specific types of PPs, namely, $f_{2}(X):= \left(X^{p^i}-X+\delta \right)^{s_1}+\left(X^{p^i}-X+\delta \right)^{s_2}+X$, and $f_{3}(X):= \left(\Tr_{m}^{n}(X)+\delta \right)^{s}+L(X)$, where $s_1, s_2, s, i, n, m$ are positive integers such that $m\mid n$ and $\delta \in \F_q$. It turns out that Zeng, Zhu, Li, and Liu~\cite{ZZLL} proposed several classes of PPs of the form $f_{2}(X)$ over finite fields of even characteristic. Subsequently, Li, Wang, Li and Zeng~\cite{LWLLZ}, as well as Liu, Xie, Liu and Zou~\cite{LXLZ}, investigated these types of PPs over finite fields of odd characteristic. In~\cite{ZTT}, Zeng, Tian and Tu explored PPs of the form $f_{3}(X)$ over finite fields of characteristic $2$, where $\delta \in \F_{2^n}$, $L(X)=\Tr_m^n(X)+X$ or $X$, 
and the exponent $s$ satisfies either of the  two conditions:
$s(2^m+1)\equiv{2^m+1}~\pmod{2^n-1}$ or $s(2^m-1)\equiv{2^m-1}\pmod{2^n-1}$.
Later, Li, Wang, Wu and Zhu~\cite{LWWZ} proposed many new classes of PPs of the forms 
$f_{4}(X):=\left(\Tr_{m}^{2m}(X)^{k_1}+\delta\right)^{s_1}+\left(\Tr_{m}^{2m}(X)^{k_2}+\delta\right)^{s_2}+X$, and $f_{5}(X):=\left(\Tr_{m}^{2m}(X)^k+\delta\right)^s+X $ over $\F_{2^{2m}}$. Furthermore, Li and Cao \cite{LX1} constructed several classes of PPs of the form $f_{1}(X)$ and $f_{2}(X)$  with the parameter $i=m$ over $\F_{p^{2m}}$. Recently, Li and Cao~\cite{LX} introduced some classes of PPs of the form $f_{4}(X)$ and $f_{5}(X)$ over $\F_{2^{2m}}$, by replacing $X$ with $L(X)=a\Tr_{m}^{2m}(X)+bX$, where $a\in \F_{2^{2m}}$ and $b \in \F_{2^m}^* $. 

Motivated by the work of several authors, as discussed above, we consider it worthwhile to explore these types of PPs further. Specifically, we focus our attention on the PPs studied by Zeng, Tian and Tu~\cite{ZTT} of the form $f_3(X)=\left(\Tr_{m}^{n}(X)+\delta \right)^{s}+L(X)$ over finite fields of even characteristic, with $L(X)$ having trivial coefficients. This raises an interesting and intriguing question of what would happen if $L(X)$ has nontrivial coefficients or if the characteristic of the finite field is not even. In this paper, we explore this question and investigate PPs of the form $f_3(X)$ over  $\F_{q^3}$, where $\F_q$ is a finite field with arbitrary characteristic, $L(X)$ is a linearized polynomial with non-trivial coefficients, and $\delta=0$. In fact, we construct six new classes of PPs of the form $f_3(X)$ over $\F_{q^3}$ by employing the technique proposed by Wu and Yuan ~\cite[Theorem 3]{WY1}, which is derived from the local method~\cite[Lemma 1]{WY1} (also see~\cite[Theorem 2.5]{Y}), and by analyzing certain equations over finite fields. It is noted in~\cite{W} that the local method \cite[Lemma 1]{WY1} is indeed a special case of AGW criteria~\cite{AGW}. We would also like to draw the reader's attention to the fact that in \cite {LHT, LWLLZ, LWWZ, LXLZ, YD, YD1, ZTT, ZZH, ZZLL, ZH1, ZH, ZYP}, the authors did not provide the compositional inverses for the PPs. However, for every new class of PPs introduced in this paper, we provide the explicit expressions for their compositional inverses. 

The rest of this paper is organized as follows. In Section~\ref{S2}, some definitions and related results are introduced. Section~\ref{S3} presents six new classes of permutation polynomials of the form $L(X)+\Tr_{m}^{3m}(X)^{s}$ over $\F_{q^3}$, where $\Tr_{m}^{3m}$ is the relative trace function from $\F_{p^{3m}}$ to $\F_{p^{m}}$, $L(X)$ is a linearized polynomial over $\F_{q^{3}}$, and $s>1$ is a positive integer. In Section~\ref{S4}, we provide the explicit expressions for the compositional inverses of the permutation polynomials obtained in Section \ref{S3} and investigate the inequivalence of these classes with the known ones of the same form. Finally, Section~\ref{S5} concludes the paper.

 \section{Preliminaries}\label{S2}
This section consists of some basic definitions and an important result that will be used in the subsequent sections.
\begin{def1}
For two positive integers $m$ and $n$ with $m\mid n$, we use $\Tr_{m}^{n}(.)$ to denote the (relative) trace function from $\F_{p^n}$ to $\F_{p^m}$, which is defined as follows
$$\Tr_{m}^{n}(X):=\displaystyle\sum_{i=0}^{\frac{n-m}{m}}X^{p^{mi}}.$$
        \end{def1}    

\begin{def1}
A polynomial of the form
$$L(X)=\displaystyle\sum_{i=0}^{n-1}\alpha_{i}X^{q^i}$$
with coefficients in an extension field  $\F_{q^n}$ of $\F_{q}$ is called $q$-linearized polynomial over $\F_{q^n}$.
\end{def1}

\begin{def1}
For a prime $p$ and for an integer $a\in \mathbb{Z}$, the $p$-adic valuation of $a$ is defined as follows
$$\nu_{p}(a):=\max\{e\in \mathbb{Z}\mid p^e \;\mbox{divides}\; a\},$$
where $\mathbb Z$ is the set of integers.
\end{def1}
\begin{def1}
For an odd prime $p$ and for $c\in \F_{p}$, we define the Legendre symbol as follows
\[
\eta(c):=\left(\frac{c}{p}\right)=
\begin{cases}
   1 & \text{if } c \text{ is quadratic residue modulo } p \text{ and } c \not\equiv 0 \pmod p, \\
    -1 & \text{if } c \text{ is nonquadratic residue modulo } p \text{ and } c \not\equiv 0 \pmod p,\\
    0 & \text{if } c \equiv 0 \pmod p.
\end{cases}
\]

\end{def1}
\begin{def1}
For a positive integer $d$, the unit circle of order $d$ is denoted by $\mu_{d}$ and is defined as
$$\mu_d:=\{X\in \overline{\F}_q : X^d=1\}, $$
where $\overline{\F}_q$ is the algebraic closure of $\F_q$.
\end{def1}

The following lemma, which is crucial to this paper and will be frequently used in Section~\ref{S3}, is stated below after defining some notations that will be used in its statement.

For positive integers $i$ with $1\leq i\leq n$, let $m_i$'s be positive integers, $u_i\in \F_{q^n}$, and  $a_i\in \F_{q^n}$ with $a_i^{q^{n-1}+q^{n-2}+\cdots+q^2+q+1}=1$. We now introduce the following notations:
$$A_i(X):=  X^{q^{n-1}}+a_iX^{q^{n-2}}+a_i^{1+q^{n-1}}X^{q^{n-3}} +\cdots  + a_i^{1+q^{n-1}+q^{n-2}+\cdots+q^2}X,$$

\begin{equation*}
D_1 := 
\begin{pmatrix}
1 & u_2 & \cdots & u_n \\
a_1^{m_1q} & u_2^{q}a_2^{m_{2}q} & \cdots & u_n^{q}a_1^{m_nq} \\
\vdots  & \vdots  & \ddots & \vdots  \\
a_1^{m_1\sum_{i=1}^{n-1}q^i} & u_2^{q^{n-1}}a_2^{m_2\sum_{i=1}^{n-1}q^i} & \cdots & u_n^{q^{n-1}}a_2^{m_n\sum_{i=1}^{n-1}q^i}
\end{pmatrix},
\end{equation*}
and
\vspace{2mm}
\begin{equation*}
D_2 := 
\begin{pmatrix}
1 & a_1 & a_1^{1+q^{n-1}} & \cdots & a_1^{1+q^{n-1}+q^{n-2}+\cdots+q^{n-j}} & \cdots & a_1^{1+q^{n-1}+q^{n-2}+\cdots+q^2} \\
1 & a_2 & a_2^{1+q^{n-1}} &\cdots & a_2^{1+q^{n-1}+q^{n-2}+\cdots+q^{n-j}} &\cdots & a_2^{1+q^{n-1}+q^{n-2}+\cdots+q^2} \\
\vdots  & \vdots  & \vdots  & \ddots & \vdots  & \ddots & \vdots  \\
1 & a_n & a_n^{1+q^{n-1}} & \cdots & a_n^{1+q^{n-1}+q^{n-2}+\cdots+q^{n-j}} & \cdots & a_n^{1+q^{n-1}+q^{n-2}+\cdots+q^2}
\end{pmatrix}.
\end{equation*}
\vspace{1mm}

\begin{lem}\label{L22} \textup{\cite[Theorem 3]{WY1}} Let $q$ be a prime power, $n>1$ be a positive integer. Let $A_i(X), D_1$ and $D_2$ be defined as above. Moreover, we assume that $m_1, m_2, \ldots, m_n$ are positive integers and $u_2, u_3, \ldots, u_n\in \F_{q^n}$. Then the polynomial
\begin{center}
    $f(X)=A_1^{m_1}(X)+\displaystyle\sum_{i=2}^{n}{u_iA_i^{m_i}(X)}$ 
\end{center}
is a PP over $\F_{q^n}$ if and only if $\gcd(m_1m_2\cdots m_n, q-1)=1$ and the determinant of the matrix $D_1D_2$ is not $0$.
\end{lem}
\section{Permutation polynomials of the form $L(X)+\Tr_{m}^{3m}(X)^{s}$}\label{S3}
In this section, we investigate the permutation behaviour of the polynomials of the form $L(X)+\Tr_{m}^{3m}(X)^{s}$ over $\F_{q^3}$ for infinitely many positive integers $s>1$ with the condition that $\gcd(s, q-1)=1$ and we propose six such classes of PPs.

Now, we prove the following two lemmas which will be used in the proof of Theorem \ref{T31}. These two lemmas essentially give an exact count for the number of solutions in $\mu_{q^2+q+1}$ of two special types of equations, where $\mu_{q^2+q+1} \subseteq  \F_{q^3}$ is the unit circle of order $q^2+q+1$. Unless otherwise stated, note that in Lemma~\ref{L31}, Lemma~\ref{L32}, and Theorem~\ref{T31}, we shall assume that $q=2^m$, where $m$ is a positive integer.
\begin{lem}\label{L31}
Let $q=2^m$, $h(X)=X^{q+1}+(A+1)X+A \in \F_{q^3}[X]$ with $A\neq 0$. Then
\begin{enumerate}
 \item[$(1)$] $h(X)$ has only one root in $\mu_{q^2+q+1}$ if $A^{q+1}+A+1\neq 0$, and
 \item[$(2)$] $h(X)$ has all $q+1$ roots in $\mu_{q^2+q+1}$ if $A^{q+1}+A+1=0$,
\end{enumerate}
where $\mu_{q^2+q+1}:=\{X\in {\F}_{q^3} : X^{q^2+q+1}=1\}. $
\end{lem}
\begin{proof} It is clear that $1$ is always a root of $h(X)$ in $\mu_{q^2+q+1}$.

\textbf{Case 1.} First, we assume that $A^{q+1}+A+1\neq 0$. Now, if possible, suppose that $\alpha \in \mu_{q^2+q+1}\setminus\{1\}$ is a root of $h(X)$, that is,
\begin{equation}\label{E6}
\alpha^{q+1}+(A+1)\alpha+A=0.
\end{equation}
By raising $q$th power to Equation~\eqref{E6}, we  obtain 
\begin{equation}\label{E7}
    \alpha^{q^2+q}+(A^q+1)\alpha^q+A^q=0.
\end{equation}
Now, multiply Equation~\eqref{E7} by $\alpha$, we get

\[
 \alpha^{q^2+q+1}+(A^q+1)\alpha^{q+1}+A^q\alpha =0 .
\]
Using the value of $\alpha^{q+1}$ from Equation~\eqref{E6} in the above equation, we obtain
\begin{equation*}
    (A^{q+1}+A+1)\alpha+(A^{q+1}+A+1)=0.
\end{equation*}
Since $A^{q+1}+A+1\neq 0$, we get $\alpha=1$, which is a contradiction.

\textbf{Case 2.} Now, suppose $A^{q+1}+A+1= 0$ and $\beta$ is any root of $h(X)$, that is,
\begin{equation}\label{E8-1}
    \beta^{q+1}+(A+1)\beta+A=0.
\end{equation}

Since $A^{q+1}+A+1= 0$ implies $A^{q^2+q+1}=1$. Thus, Equation~\eqref{E8-1} can be written as 
\begin{equation*}
       A^{q^2+q+1}\beta^{q+1}+A^{q+1}\beta+A=0
\end{equation*}
which implies
\begin{equation*}
    (A^q\beta)^{q+1}+A^q\beta+1=0,
\end{equation*}
as $A\neq 0.$ So, we have $A^q\beta \in \mu_{q^2+q+1}$. Hence  $\beta\in \mu_{q^2+q+1}$ as $A\in \mu_{q^2+q+1}$, which completes the proof.

\end{proof}

\begin{lem}\label{L32}
Let $q=2^m$, $m$ be odd. Then $X^{q+1}+(A+1)X^{q}+A=0$ has no solution in $\mu_{q^2+q+1}\setminus\{1\}\subseteq \F_{q^3}$, where $A\in \F_{q}
^*$.
\end{lem}
\begin{proof} Let $h(X)=X^{q+1}+(A+1)X^{q}+A$. It is clear that $h(1)=0$. Now, suppose $\alpha \in \mu_{q^2+q+1}\setminus\{1\}$ such that $h(\alpha)=0$, that is,
\begin{equation}\label{E12}
\alpha^{q+1}+(A+1)\alpha^q+A=0.
\end{equation}
We raise power $q^2$ to Equation~\eqref{E12} to get
\begin{equation}\label{E13}
(A+1)\alpha^{q+1}+A\alpha^q+1=0.
\end{equation}
Now by adding Equation~\eqref{E12} and Equation~\eqref{E13}, we obtain
\begin{equation}\label{E14}
    \alpha^{q+1}+\frac{\alpha^q}{A}+1+\frac{1}{A}=0.
\end{equation}
We now add Equation~\eqref{E12} and Equation~\eqref{E14} so as to get
\begin{equation*}
    \alpha^{q}(A+1+\frac{1}{A})+A+1+\frac{1}{A}=0
\end{equation*}
which implies
\begin{equation}\label{E15}
    \alpha^{q}(A^2+A+1)+A^2+A+1=0.
\end{equation}
Note that $A^2+A+1\neq 0$ as $A\in\F_{2^m}^*$ and $m$ is  odd.
Hence Equation~\eqref{E15} implies $\alpha=1$, which is a contradiction, and we are done.
\end{proof}
\begin{rmk}\label{rmk31}
It follows from the proof of Lemma~\ref{L32} that if $q=2^m$, where $m$ is any positive integer, and if  $A\in \F_{q}^*$ such that $A^2+A+1\neq 0$, then $X^{q+1}+(A+1)X^{q}+A=0$ has no solution in $\mu_{q^2+q+1}\setminus\{1\}\subseteq \F_{q^3}$.    
\end{rmk}
By using Lemma~\ref{L31} and Lemma~\ref{L32}, we discuss the permutation behaviour of a class of polynomials over finite fields of even characteristic in the following theorem. 
\begin{thm}\label{T31}
Let $k,m$ be positive integers such that $\nu_{3}(k)\leq \nu_{3}(m)$, $q=2^m$, $a\in \mu_{q^2+q+1}  \subseteq \F_{q^3}$, and 
\begin{center}
$f_1(X)=(a+a^{2^k})X^q+(a^{1+q^2}+a^{2^{k}(1+q^2)})X+\Tr_{m}^{3m}(X)^s$  
\end{center} be a polynomial in $\F_{q^3}[X]$. Then
\begin{enumerate}
    \item[$(1)$] if $q\equiv 2\pmod{3}$, $f_1(X)$ is a PP if and only if $a\neq 1$ and $\gcd(s, q-1)=1$, 
   \item[$(2)$]if $q\equiv 1\pmod{3}$ and if $2\nmid k$, $f_1(X)$ is a PP if and only if $a\neq 1$ and $\gcd(s, q-1)=1$.
\end{enumerate}
\end{thm}
\begin{proof}
We first express $f_1(X)$ into the form as that of $f(X)$ given in Lemma~\ref{L22}. Then the determinants $\det D_1$ and $\det D_2$ corresponding to $f_1(X)$ are given by
\begin{equation*}
\det D_1=a^{(2^k+1)q}a^{2^kq^2}+a^q+a^{2^kq}+a^{(2^k+1)q}a^{q^2}+a^{q+q^2}+a^{2^kq}a^{2^kq^2},
\end{equation*}
 and 
\begin{equation*}
    \det D_2=a^{(2^k+1)}a^{2^kq^2}+a+a^{2^k}+a^{(2^k+1)}a^{q^2}+a^{1+q^2}+a^{2^k}a^{2^kq^2}.
\end{equation*}

\textbf{Case 1.}  Let $q\equiv 2\pmod{3}$. Assume that $f_1(X)$ is a PP over $\F_{q^3}$, we have to show that $a\neq 1$ and $\gcd(s, q-1)=1$. If possible, suppose that $a=1$, we get $\det D_1=0$ and $\det D_2=0$. Therefore, by Lemma~\ref{L22}, $f_1(X)$ is not a PP over $\F_{q^3}$. Now, if $\gcd(s, q-1)\neq 1$, then again by Lemma~\ref{L22}, $f_1(X)$ is not a PP over $\F_{q^3}$. Hence, $a\neq 1$ and $\gcd(s, q-1)=1.$

 Conversely, suppose that $a\neq 1$ and $\gcd(s, q-1)=1$. In view of Lemma~\ref{L22}, it is enough to show that $\det D_1\neq0$ and $\det D_2\neq0$. Suppose on the contrary, that $\det D_1=0$, that is,
 \begin{equation*}
a^{(2^k+1)q}a^{2^kq^2}+a^q+a^{2^kq}+a^{(2^k+1)q}a^{q^2}+a^{q+q^2}+a^{2^kq}a^{2^kq^2}=0.
\end{equation*}
By substituting $x=a^q$ and $y=a^{q^2}$ in the above expression, we obtain $x^{2^k}y^{2^k}+1+x^{2^k-1}+x^{2^k}y+y+x^{2^k-1}y^{2^k}=0$. We now add $x^{2^k}$ on both sides so as to get $x^{2^k-1}y^{2^k}(x+1)+y(x+1)^{2^k}+x^{2^k-1}(x+1)+(x+1)^{2^k}=0$, which is equivalent to  $x^{2^k-1}y^{2^k}+y(x+1)^{2^k-1}+x^{2^k-1}+(x+1)^{2^k-1}=0$ as $x\neq 1$, and after simplification, we obtain $x^{2^k-1}(y+1)^{2^k}+(x+1)^{2^k-1}(y+1)=0.$ Further, since $x\neq 1$ and $y\neq 1$, we obtain $\left(\frac{x(y+1)}{x+1}\right)^{2^k-1}=1.$ Also, $\left(\frac{x(y+1)}{x+1}\right)^{2^{3m}-1}=1$ as $q=2^m$. Now the last two equations together would imply that $\left(\frac{x(y+1)}{x+1}\right)^{2^{d}-1}=1$, where $d=\gcd(k, 3m)$. It is clear that $\frac{x(y+1)}{x+1}$ is a $(2^d-1)$-th root of unity, and there are two possible cases, namely, either $\frac{x(y+1)}{x+1}=1$ or else we take $\frac{x(y+1)}{x+1}=\omega$ with $\omega\neq 1$. It is straightforward to see that $d=\gcd(k, 3m)=\gcd(k, m)$ as $\nu_{3}(k)\leq \nu_{3}(m)$, and as a consequence, it follows that $\omega \in \F_{q}$ as $(2^d-1)| (q-1)$.
If $\frac{x(y+1)}{x+1}=1$, then $a^{q+q^2}=1$, which leads to a contradiction. Now, if $\frac{x(y+1)}{x+1}=\omega$, we have 
\begin{equation}\label{E19}
    a^{q+1}+\frac{\omega}{\omega+1}a+\frac{1}{\omega+1}=0.
\end{equation}
Thus, by letting $A=\frac{1}{\omega+1}\in \F_{q}$, it is easy to see that Equation~\eqref{E19} implies that $a\neq1$ is a solution to the following equation 
\begin{equation}\label{e19}
X^{q+1}+(A+1)X+A=0. 
\end{equation}
Since $A^{q+1}+A+1=A^2+A+1$ and $A^2+A+1\neq 0$ as $q\equiv 2\pmod{3}$, it follows from Lemma~\ref{L31} that Equation~\eqref{e19} has no root in $\mu_{q^2+q+1}\setminus\{1\}$, which gives a contradiction. Thus, $\det D_1\neq 0$.

It remains to show that $\det D_2 \neq 0$. Suppose on the contrary, that $\det D_2=0$, that is,
\begin{equation*}
    a^{(2^k+1)}a^{2^kq^2}+a+a^{2^k}+a^{(2^k+1)}a^{q^2}+a^{1+q^2}+a^{2^k}a^{2^kq^2}=0.
\end{equation*}
Now, by substituting $x=a$, $y=a^{q^2}$, and by adding $x^{2^k}$ on both sides, we get 
\begin{equation*}
    \left(\frac{x(y+1)}{x+1}\right)^{2^k-1}=1.
\end{equation*}
Following the similar arguments as above, we would either have $\frac{x(y+1)}{x+1}=1$ or else we take $\frac{x(y+1)}{x+1}=\omega$, where $\omega\neq 1$, $\omega^{2^d-1}=1$, and $\omega \in \F_{q}$. If $\frac{x(y+1)}{x+1}=1$, then $a^{1+q^2}=1$, which gives rise to a contradiction. If $\frac{x(y+1)}{x+1}=\omega$, we have
\begin{equation}\label{E20}
    a^{q+1}+ (A+1)a^q+A=0,
\end{equation}
where $A=\frac{1}{\omega+1}\in \F_{q}$. Since $q\equiv 2\pmod{3}$, it follows from Lemma \ref{L32} that $X^{q+1}+(A+1)X^{q}+A=0$ has no solution in $\mu_{q^2+q+1}\setminus\{1\}$, which is obviously a contradiction to Equation~\eqref{E20}. Hence, $\det D_2 \neq 0$.

\textbf{Case 2.} Let $q\equiv 1\pmod{3}$. First, we consider $f_1(X)$ is a PP over $\F_{q^3}$. If possible, suppose that $a=1$ then we get $\det D_1=0$ and $\det D_2=0$. Therefore, according to Lemma \ref{L22}, $f_1(X)$ is not a PP over $\F_{q^3}$. Further, if we assume that
 $\gcd(s, q-1)\neq 1$, then again by Lemma \ref{L22}, $f_1(X)$ is not a PP over $\F_{q^3}$. Hence, $a \neq 1$ and $\gcd(s, q-1)=1$.

 Conversely, we assume that $a\neq 1$ and $\gcd(s, q-1)=1$. We shall now show that $\det D_1\neq 0$ and $\det D_2\neq 0$. If possible, suppose that $\det D_1=0$. Following similar arguments as in the Case 1, we get Equation \eqref{E19}, which, in turn, ensures that $a\neq 1$ is a solution to Equation~\eqref{e19}. However, in view of Lemma \ref{L31}, the Equation \eqref{e19} has a solution in $\mu_{q^2+q+1}\setminus\{1\}$ only if  $A^{2}+A+1=0$, where $A=\frac{1}{\omega+1}$ with $\omega$ satisfying $\omega^{2^k-1}=1$. Note that $A^2+A+1=0$ if and only if $\omega^2+\omega+1=0$. But $\omega^2+\omega+1=0$ implies that $\omega^3=1$. Thus, we obtain $\omega^{\gcd(3, 2^k-1)}=1$. Since $2\nmid k$, we have $\gcd(3, 2^k-1)=1$. This means, $\omega=1$, which is a contradiction. Therefore, $\det D_1\neq 0$. Next, if possible, we assume that $\det D_2=0$. As in the Case 1, $\det D_2=0$ will eventually lead to the Equation \eqref{E20}, which implies that $a\in\mu_{q^2+q+1}\setminus\{1\}$ is solution to $X^{q+1}+(A+1)X^{q}+A=0$, where $A=\frac{1}{\omega+1}$ with $\omega^{2^k-1}=1$. According to Remark~\ref{rmk31}, $X^{q+1}+(A+1)X^{q}+A=0$ has a solution in $\mu_{q^2+q+1}\setminus\{1\}$ only if $A^{2}+A+1=0$. Now, similar arguments as above will give us $\omega=1$, which is a contradiction. Hence, $\det D_2\neq 0$.
\end{proof}

In the next three theorems, we investigate three classes of  PPs with necessary and sufficient conditions on the coefficients of polynomials of the type $L(X)+\Tr_{m}^{3m}(X)^{s}$ over $\F_{q^3}$, where $q=p^m$, $p$ is an odd prime, and $m$ is some positive integer. In the proofs of these theorems, we use the non-existence of solutions to certain equations in two variables over finite fields as a tool.

\begin{thm}\label{T32}
Let $q=p^m$, $p$ is an odd prime, $a\in \mu_{q^2+q+1}\ \subseteq \F_{q^3}$, and $f_2(X)=2X^{q^2}+(a+a^q)X^q+(a^{1+q^2}+a^{1+q})X+\Tr_{m}^{3m}(X)^s \in \F_{q^3}[X]$. Then
\begin{enumerate}
    \item[$(1)$] if $q\equiv 0\pmod{3}$, $f_2(X)$ is a PP if and only if $a$ is not a solution to the equations $X^{q+1}+X+1=0$ and $X^{q+1}+X^q+1=0$, and $\gcd(s, q-1)=1$,
    \item[$(2)$] if $q\equiv 1\pmod{3}$, $f_2(X)$ is a PP if and only if $a$ is not a solution to the  equations $X^{q+1}+\frac{1+\alpha}{1-\alpha}X-\frac{2}{1-\alpha}=0$, $X^{q+1}+\frac{1-\alpha}{1+\alpha}X-\frac{2}{1+\alpha}=0$, $2X^{q+1}-(1-\beta)X^q-(1+\beta)=0$ and $2X^{q+1}-(1+\beta)X^q-(1-\beta)=0$, where $\alpha, \beta \in \mathbb F_{q^3}$ such that $\alpha^2=\beta^2=-3$, and $\gcd(s, q-1)=1$,
    \item[$(3)$] if $q\equiv 2\pmod{3}$, $f_2(X)$ is a PP if and only if $a\neq 1$ and $\gcd(s, q-1)=1$.
    \end{enumerate}
    
\end{thm}
\begin{proof}
    First, we transform $f_2(X)$ into the form specified in Lemma~\ref{L22}. Then, the determinants $\det D_1$ and $\det D_2$ corresponding to $f_2(X)$ are
    \begin{equation*}
\det D_1=1-a^q+a^{q^2}-a^{2q^2+q}+a^{q+q^2}-a^{q^2+1},
\end{equation*}
and
\begin{equation*}
    \det D_2=a^{q+2}-a+a^{q}-1+a^{1+q^2}-a^{q+1}.
\end{equation*}

\textbf{Case 1.} Suppose that $q\equiv 0\pmod{3}$. First, we assume that $f_2(X)$ is a PP over $\F_{q^3}.$ If possible, let $\gcd(s, q-1)\neq 1$, then by Lemma \ref{L22}, $f_2(X)$ is not a PP over $\F_{q^3}.$ Now, suppose that $a$ is a solution to either $X^{q+1}+X+1=0$ or $X^{q+1}+X^q+1=0$.
First, assume that $a$ is a solution to $X^{q+1}+X+1=0$, that is, $a^{q+1}=-a-1$, which implies $a^{q+q^2}=-a^q-1$. Using this, the expression for the $\det D_1$ is given by
\begin{align*}
    \det D_1=&1-a^q+a^{q^2}-a^{2q^2+q}+a^{q+q^2}-a^{q^2+1}\\
    =&\frac{a^q-a^{2q}+a^{q+q^2}-a^{2q^2+2q}+a^{2q+q^2}-1}{a^q}\\
    =&\frac{a^q-a^2q-a^q-1+(a^q+1)^2+a^q(-a^q-1)-1}{a^q}.
\end{align*}
After simplifying and using $q\equiv 0\pmod{3}$, we get $\det D_1=0$. So by Lemma \ref{L22}, $f_2(X)$ is not a PP over $\F_{q^3}.$ 

Now, we suppose that $a$ is a solution to $X^{q+1}+X^q+1=0$. Then $a^{q+1}=-a^q-1$. Using this, $\det D_2$ can written as follows
 \begin{align*}
     \det D_2&=a^{q+2}-a+a^{q}-1+a^{1+q^2}-a^{q+1}\\
     &=\frac{ a^{2q+2}-a^{q+1}+a^{2q}-a^{q}+1-a^{2q+1}}{a^q}\\
    & =\frac{a^{2q}+1-a^q+a^q+1+a^{2q}-a^q+1+a^{2q}+a^{q}}{a^q}.
 \end{align*}
Since $q\equiv 0\pmod{3}$, so we get $\det D_2=0$. Again by Lemma \ref{L22}, $f_2(X)$ is not a PP over $\F_{q^3}$, which is a contradiction.

Conversely, suppose that $a$ is not a solution to any of the equations $X^{q+1}+X+1=0$, $X^{q+1}+X^q+1=0$, and $\gcd(s,q-1)=1$. By Lemma \ref{L22}, it is enough to show that $\det D_1\neq 0$ and $\det D_2\neq 0.$ On the contrary, we assume that $\det D_1=0$, that is, $1-a^q+a^{q^2}-a^{2q^2+q}+a^{q+q^2}-a^{q^2+1}=0$, which is equivalent to $a^{q}-a^{2q}+a^{q+q^2}-a^{2q+2q^2}+a^{2q+q^2}-1=0.$ After substituting $x=a^q$ and $y=a^{q^2}$, one arrives at $$y=\frac{\left(\frac{1}{x}+1\right)\pm\sqrt{-3}\left(\frac{1}{x}-1\right)}{2}.$$ 
Since $q\equiv 0\pmod{3}$, so we obtain $2xy=1+x$, where  $x=a^q$ and $y=a^{q^2}$. This implies that $a$ is a solution of $X^{q+1}+X+1=0$, which is a contradiction. Hence $\det D_1\neq0.$ 

Next, if possible, suppose that $\det D_2=0$, that is, $a^{q+2}-a+a^{q}-1+a^{1+q^2}-a^{q+1}=0$, which is equivalent to $a^{2q+2}-a^{q+1}+a^{2q}-a^{q}+1-a^{2q+1}=0.$ By replacing $a$ with $x$ and $a^q$ with $y$, we will get $$x=\frac{\left(\frac{1}{y}+1\right)\pm\sqrt{-3}\left(\frac{1}{y}-1\right)}{2}.$$ As $q\equiv 0\pmod{3}$, we get 2xy=1+y. It yields $a^{q+1}+a^{q}+1=0$, which is not true. Therefore, $\det D_2\neq0$.

\textbf{Case 2.} Let $q\equiv 1\pmod{3}$. Suppose that $f_2(X)$ is a PP over $\F_{q^3}.$ If possible, let $\gcd(s, q-1)\neq 1$, then by Lemma \ref{L22}, $f_2(X)$ is not a PP over $\F_{q^3}.$ Now, consider  $a$ is a solution to any of the equations $X^{q+1}+\frac{1+\alpha}{1-\alpha}X-\frac{2}{1-\alpha}=0$, $X^{q+1}+\frac{1-\alpha}{1+\alpha}X-\frac{2}{1+\alpha}=0$, $2X^{q+1}-(1-\beta)X^q-(1+\beta)=0$, or $2X^{q+1}-(1+\beta)X^q-(1-\beta)=0$, where $\alpha$, $\beta \in \F_{q^3}$ with $\alpha^2=-3=\beta^2$.

Firstly, we consider $a$ is a solution of the equations $X^{q+1}+\frac{1+\alpha}{1-\alpha}X-\frac{2}{1-\alpha}=0$ or $X^{q+1}+\frac{1-\alpha}{1+\alpha}X-\frac{2}{1+\alpha}=0$, that is, $a^q=\frac{2}{(1-\alpha)a}-\frac{(1+\alpha)}{1-\alpha}$ or $a^q=\frac{2}{(1+\alpha)a}-\frac{(1-\alpha)}{1+\alpha}$. Using $a^q=\frac{2}{(1-\alpha)a}-\frac{(1+\alpha)}{1-\alpha}$ or $a^q=\frac{2}{(1+\alpha)a}-\frac{(1-\alpha)}{1+\alpha}$ in the following expression
\begin{equation*}
    \det D_1=1-a^q+a^{q^2}-a^{2q^2+q}+a^{q+q^2}-a^{q^2+1}=\frac{a^{q}-a^{2q}+\frac{1}{a}-\frac{1}{a^2}+\frac{a^q}{a}-1}{a^q},
\end{equation*}  
 we get $\det D_1=0$.

Secondly, by taking $a$ to be a solution of $2X^{q+1}-(1-\beta)X^q-(1+\beta)=0$ or $2X^{q+1}-(1+\beta)X^q-(1-\beta)=0$, we get $a^{q+1}=\frac{(1-\alpha)a^q+(1+\alpha)}{2}$ or $a^{q+1}=\frac{(1+\alpha)a^q+(1-\alpha)}{2}$, respectively.  Now, using these, in the following expression
    \begin{align*}
    \det D_2 &=a^{q+2}-a+a^{q}-1+a^{1+q^2}-a^{q+1}=\frac{a^{2q+2}-a^{q+1}+a^{2q}-a^{q}+1-a^{2q+1}}{a^q},
    \end{align*}
   we get $\det D_2=0$.
  Thus, by Lemma~\ref{L22}, $f_2 (X)$ is not a PP over $\mathbb F_{q^3}$, which is a contradiction.  
  
 Conversely, suppose that $\gcd(s, q-1)=1$ and $a$ is not a solution to any of the equations $X^{q+1}+\frac{1+\alpha}{1-\alpha}X-\frac{2}{1-\alpha}=0$, $X^{q+1}+\frac{1-\alpha}{1+\alpha}X-\frac{2}{1+\alpha}=0$, $2X^{q+1}-(1-\beta)X^q-(1+\beta)=0$, $2X^{q+1}-(1+\beta)X^q-(1-\beta)=0$, where $\alpha$, $\beta \in \F_{q^3}$ with $\alpha^2=-3=\beta^2$. By Lemma \ref{L22}, it is enough to show that $\det D_1\neq 0$ and $\det D_2\neq 0.$ 

Now, if possible, we assume that $\det D_1=0$. Then by following similar arguments as in Case 1, we get 
\begin{equation*}
    y=\frac{\left(\frac{1}{x}+1\right)\pm\sqrt{-3}\left(\frac{1}{x}-1\right)}{2}.
\end{equation*}
Since $q\equiv 1\pmod{3}$, so there exist some $\alpha\in \F_{q^3}$ such that $\alpha^2=-3$. Using this $\alpha$ in the above equation, we have $2xy=x+1+\alpha-\alpha x$ or $2xy=x+1-\alpha+\alpha x$, where $x=a^q$ and $y=a^{q^2}$. Now, it follows that $a$ is a solution to $X^{q+1}+\frac{1+\alpha}{1-\alpha}X-\frac{2}{1-\alpha}=0$ or $X^{q+1}+\frac{1-\alpha}{1+\alpha}X-\frac{2}{1+\alpha}=0$, which is a contradiction. 

Next, if possible, we consider that $\det D_2=0$. Then, we have
\begin{equation*}
    x=\frac{\left(\frac{1}{y}+1\right)\pm\sqrt{-3}\left(\frac{1}{y}-1\right)}{2},
\end{equation*}
where $x=a$ and $y=a^q$. Thus, we get $2xy=1+y+\beta-\beta y$ or $2xy=1+y-\beta+\beta y$ for some $\beta\in \F_{q^3}$ such that $\beta^2=-3$. This will imply that $a$ is a solution to $2X^{q+1}-(1-\beta)X^q-(1+\beta)=0$ or $2X^{q+1}-(1+\beta)X^q-(1-\beta)=0$, which is a contradiction. Hence, $\det D_1 \neq 0$ and $\det D_2 \neq 0$. 

\textbf{Case 3.} Let $q\equiv 2\pmod{3}$. We first consider $f_2(X)$ is a PP over $\F_{q^3}$. If possible, suppose that $a=1$ then we get $\det D_1=0$ and $\det D_2=0$. Therefore, according to Lemma~\ref{L22}, $f_2(X)$ is not a PP over $\F_{q^3}$. Now, if
 $\gcd(s, q-1)\neq 1$, then again by Lemma~\ref{L22}, $f_2(X)$ is not a PP over $\F_{q^3}$. Hence, $a\neq 1$ and $\gcd(s, q-1)=1.$ 

 Conversely, suppose that $a\neq 1$ and $\gcd(s, q-1)=1$. By Lemma \ref{L22}, it is sufficient to show that $\det D_1\neq 0$ and $\det D_2\neq 0$. If possible, let us assume that $\det D_1=0$, Then by the similar arguments as in Case 1, we get 
\begin{equation*}
    y=\frac{\left(\frac{1}{x}+1\right)\pm\sqrt{-3}\left(\frac{1}{x}-1\right)}{2}.
\end{equation*} 
Since $x\neq1$ and $-3$ has no square root in $\F_{q^3}$ as $q\equiv 2\pmod{3}$. This implies $y=a^{q^2}\notin \F_{q^3}$, which is a contradiction. Hence $\det D_1\neq 0.$

Next, if possible, suppose that $\det D_2=0$, we will get $$x=\frac{\left(\frac{1}{y}+1\right)\pm\sqrt{-3}\left(\frac{1}{y}-1\right)}{2}.$$ Now, by similar arguments as before, we obtain that $x=a \notin \mathbb F_{q^3}$, a contradiction. Thus,  $\det D_2 \neq 0$.
 \end{proof}

\begin{thm}\label{T33}
Let $q=p^m$, $p$ is an odd prime, $a\in \mu_{q^2+q+1} \subseteq \F_{q^3}$, and $f_3(X)=2X^{q^2}+(a+a^{2+q})X^q+(a^{1+q^2}+a^{2+q^2})X+\Tr_{m}^{3m}(X)^s \in \F_{q^3}[X]$. Then
\begin{enumerate}
    \item [$(1)$] if $q\equiv 0\pmod{3}$, $f_3(X)$ is a PP if and only if $a$ is not a solution to the equations $X^{q+1}+X^q+1=0$ and $X^{q+1}+X+1=0$, and $\gcd(s, q-1)=1$,
    
    \item [$(2)$]if $q\equiv 1\pmod{3}$, $f_3(X)$ is a PP if and only if $a$ is not a solution to the  equations $X^{q+1}-\frac{2}{1+\alpha}X^q+\frac{1-\alpha}{1+\alpha}=0$, $X^{q+1}-\frac{2}{1-\alpha}X^q+\frac{1+\alpha}{1-\alpha}=0$, $X^{q+1}-\frac{2}{1+\beta}X+\frac{1-\beta}{1+\beta}=0$ and $X^{q+1}-\frac{2}{1-\beta}X+\frac{1+\beta}{1-\beta}=0$,  where $\alpha$, $\beta \in \F_{q^3}$ with $\alpha^2=-3=\beta^2$, and $\gcd(s, q-1)=1$,
      \item [$(3)$] if $q\equiv 2\pmod{3}$, $f_3(X)$ is a PP if and only if $a\neq 1$ and $\gcd(s, q-1)=1$.
\end{enumerate} 
   \end{thm}
\begin{proof}
    First, we convert $f_3(X)$ into the form provided in Lemma~\ref{L22}. Then the determinants $\det D_1$ and $\det D_2$ corresponding to $f_3(X)$ are given by
    \begin{equation*}
\det D_1=a^{2q+2q^2}-a^q+a^{2q+q^2}-a^{3q+2q^2}+a^{q+q^2}-a^{q+2q^2},
\end{equation*}
and
\begin{equation*}
    \det D_2=a^{3+q^2}-a+a^{2+q}-a^2+a^{1+q^2}-a^{2+q^2}.
\end{equation*}

\textbf{Case 1.} Let $q\equiv 0\pmod{3}$. Suppose that $f_3(X)$ is a PP over $\F_{q^3}.$ If possible, let $\gcd(s, q-1)\neq 1$, then by Lemma \ref{L22}, $f_3(X)$ is not a PP over $\F_{q^3}.$ Now, consider $a$ is a solution of the equations $X^{q+1}+X^q+1=0$ or $X^{q+1}+X+1=0$. If $a$ is a solution of  $X^{q+1}+X^q+1=0$, then  $a^{q+1}=-a^q-1$. Now using $a^{q+1}=-a^q-1$ and $q\equiv 0\pmod{3}$, in the last expression of the following equation
\begin{equation*}
    \det D_1=a^{2q+2q^2}-a^q+a^{2q+q^2}-a^{3q+2q^2}+a^{q+q^2}-a^{q+2q^2}=\frac{a^q-a^{2q+2}+a^{2q+1}-a^{2q}+a^{q+1}-1}{a^{q+2}},
\end{equation*}
we get $\det D_1=0$. 

If $a$ is a solution to $X^{q+1}+X+1=0$, we have $a^{q+1}=-a-1$. Further using $a^{q+1}=-a-1$ and $q\equiv 0\pmod{3}$ in the following expression of $\det D_2$ 
\begin{equation*}
     \det D_2=a^{3+q^2}-a+a^{2+q}-a^2+a^{1+q^2}-a^{2+q^2}=\frac{ a^3-a^{q+2}+a^{2q+3}-a^{q+3}+a-a^{2}}{a^{q+1}},
 \end{equation*}
 we obtain $\det D_2=0$. Therefore, according to Lemma~\ref{L22}, $f_3(X)$ is not a PP over $\F_{q^3}$, which is a contradiction.

 Conversely, suppose that $\gcd(s, q-1)=1$ and $a$ is not a solution of any of the  equations $X^{q+1}+X^q+1=0$ and $X^{q+1}+X+1=0$. So by Lemma \ref{L22}, it is enough to show $\det D_1\neq 0$ and $\det D_2\neq 0.$ On the contrary, suppose that $\det D_1=0$, we have $a^{2q+2q^2}-a^q+a^{2q+q^2}-a^{3q+2q^2}+a^{q+q^2}-a^{q+2q^2}=0$, which is same as $a^{q+2q^2}-1+a^{q+q^2}-a^{2q+2q^2}+a^{q^2}-a^{2q^2}=0.$ By substituting $x=a^q$ and $y=a^{q^2}$, we get
\begin{equation*}
    x=\frac{\left(\frac{1}{y}+1\right)\pm\sqrt{-3}\left(\frac{1}{y}-1\right)}{2}.
\end{equation*}
Since $q\equiv 0\pmod{3}$, so  we get $2xy=1+y$, where $x=a^q$ and $y=a^{q^2}$. This implies that $a$ is a solution to $X^{q+1}+X^q+1=0$, which is a contradiction. 

Next, if $\det D_2=0$, we have  $a^{3+q^2}-a+a^{2+q}-a^2+a^{1+q^2}-a^{2+q^2}=0$, which is equivalent to $a^{3+2q^2}-a^{1+q^2}+a-a^{2+q^2}+a^{1+2q^2}-a^{2+2q^2}=0.$ Now by taking $x=a$ and $y=a^{q^2}$, we arrive at
\begin{equation*}
    x=\frac{\left(\frac{1}{y}+1\right)\pm\sqrt{-3}\left(\frac{1}{y}-1\right)}{2}.
\end{equation*} 
As $q\equiv 0\pmod{3}$, so  we get $2xy=1+y$. This forces that $a$ is a solution to $X^{q+1}+X+1=0$, which is not true. Hence, $\det D_1\neq 0$ and  $\det D_2\neq 0.$ 

\textbf{Case 2.} Assume that $q\equiv 1\pmod{3}$.  First, let us suppose that $f_3(X)$ is a PP over $\F_{q^3}.$ If possible, we consider that $\gcd(s, q-1)\neq 1$, then by Lemma \ref{L22}, $f_3(X)$ is not a PP over $\F_{q^3}.$ Now, consider $a$ is a solution of any of the equations  $X^{q+1}-\frac{2}{1+\alpha}X^q+\frac{1-\alpha}{1+\alpha}=0$ or $X^{q+1}-\frac{2}{1-\alpha}X^q+\frac{1+\alpha}{1-\alpha}=0$ or $X^{q+1}-\frac{2}{1+\beta}X+\frac{1-\beta}{1+\beta}=0$ or $X^{q+1}-\frac{2}{1-\beta}X+\frac{1+\beta}{1-\beta}=0$,  where $\alpha$, $\beta \in \F_{q^3}$ with $\alpha^2=-3=\beta^2$.

If $a$ is a solution to $X^{q+1}-\frac{2}{1+\alpha}X^q+\frac{1-\alpha}{1+\alpha}=0$ or $X^{q+1}-\frac{2}{1-\alpha}X^q+\frac{1+\alpha}{1-\alpha}=0$, then we have $a^{q+1}=\frac{2}{(1+\alpha)}a^q-\frac{(1-\alpha)}{1+\alpha}$ or $a^{q+1}=\frac{2}{(1-\alpha)}a^q-\frac{(1+\alpha)}{1-\alpha}$. Using these into the later expression of the following equation
\begin{equation*}
    \det D_1=a^{2q+2q^2}-a^q+a^{2q+q^2}-a^{3q+2q^2}+a^{q+q^2}-a^{q+2q^2}=\frac{a^{q}-a^{2q+2}+a^{2q+1}-a^{2q}+a^{q+1}-1}{a^{q+2}},
\end{equation*}
we get $\det D_1=0$. 

Next, if $a$ is a solution to $X^{q+1}-\frac{2}{1+\beta}X+\frac{1-\beta}{1+\beta}=0$ or $X^{q+1}-\frac{2}{1-\beta}X+\frac{1+\beta}{1-\beta}=0$, we have $a^{q+1}=\frac{2}{(1-\alpha)}a-\frac{(1+\alpha)}{1-\alpha}$ or $a^q=\frac{2}{(1+\alpha)}a-\frac{(1-\alpha)}{1+\alpha}$. Now, the $\det D_2$ as expressed below
\begin{equation*}
    \det D_2=a^{3+q^2}-a+a^{2+q}-a^2+a^{1+q^2}-a^{2+q^2}=\frac{a^{2}-a^{q+1}+a^{2+2q}-a^{2+q}+1-a}{a^q},
\end{equation*}
can be made $0$ by using $a^{q+1}=\frac{2}{(1-\alpha)}a-\frac{(1+\alpha)}{1-\alpha}$ or $a^q=\frac{2}{(1+\alpha)}a-\frac{(1-\alpha)}{1+\alpha}$. Therefore, according to Lemma~\ref{L22}, $f_3(X)$ is not a PP over $\F_{q^3}$, which is a contradiction.

Conversely,  suppose that $\gcd(s, q-1)=1$ and $a$ is not a solution of any of the equations $X^{q+1}-\frac{2}{1+\alpha}X^q+\frac{1-\alpha}{1+\alpha}=0$, $X^{q+1}-\frac{2}{1-\alpha}X^q+\frac{1+\alpha}{1-\alpha}=0$, $X^{q+1}-\frac{2}{1+\beta}X+\frac{1-\beta}{1+\beta}=0$ and $X^{q+1}-\frac{2}{1-\beta}X+\frac{1+\beta}{1-\beta}=0$,  where $\alpha$, $\beta \in \F_{q^3}$ with $\alpha^2=-3=\beta^2$. We aim to show that $\det D_1\neq 0$ and $\det D_2\neq 0$. If possible, suppose that $\det D_1=0$, we have 
\begin{equation*}
     x=\frac{\left(\frac{1}{y}+1\right)\pm\sqrt{-3}\left(\frac{1}{y}-1\right)}{2}.
\end{equation*}
Since $q\equiv 1\pmod{3}$, so there exist some $\alpha\in \F_{q^3}$ such that $\alpha^2=-3$. Therefore from the above equation, we get $2xy=1+y+\alpha-\alpha y$ or $2xy=1+y-\alpha+\alpha y$, where $x=a^q$ and $y=a^{q^2}$. This yields that $a$ is a solution of $X^{q+1}-\frac{2}{1+\alpha}X^q+\frac{1-\alpha}{1+\alpha}=0$ or $X^{q+1}-\frac{2}{1-\alpha}X^q+\frac{1+\alpha}{1-\alpha}=0$, which is a contradiction.  

Now, if $\det D_2=0$, then we have 
\begin{equation*}
    x=\frac{\left(\frac{1}{y}+1\right)\pm\sqrt{-3}\left(\frac{1}{y}-1\right)}{2},
\end{equation*}
where $x=a$, $y=a^{q^2}$. Thus, we get $2xy=1+y+\beta-\beta y$ or $2xy=1+y-\beta+\beta y$ for some $\beta \in \F_{q^3}$ such that $\beta^2=-3$. Therefore, we obtain $a$ as a solution of  $X^{q+1}-\frac{2}{1+\beta}X+\frac{1-\beta}{1+\beta}=0$ or $X^{q+1}-\frac{2}{1-\beta}X+\frac{1+\beta}{1-\beta}=0$, which is a contradiction. Hence, $\det D_1\neq 0$ and $\det D_2\neq 0$. 

\textbf{Case 3.}  Suppose that $q\equiv 2\pmod{3}$. Assume that  $f_3(X)$ is a PP over $\F_{q^3}$. On the contrary, suppose that $a=1$ then we get $\det D_1=0$ and $\det D_2=0$. Therefore, according to Lemma~\ref{L22}, $f_3(X)$ is not a PP over $\F_{q^3}$. Now, if $\gcd(s, q-1)\neq 1$, then again by Lemma~\ref{L22}, $f_3(X)$ is not a PP over $\F_{q^3}$. Hence $a\neq 1$ and $\gcd(s, q-1)=1.$

 Conversely, suppose that $a\neq 1$ and $\gcd(s, q-1)=1$. From Lemma \ref{L22}, it is sufficient to show that $\det D_1\neq 0$ and $\det D_2\neq 0$. If possible, assume that $\det D_1=0$, we have 

\begin{equation*}
    x=\frac{\left(\frac{1}{y}+1\right)\pm\sqrt{-3}\left(\frac{1}{y}-1\right)}{2}.
\end{equation*}
Since $y\neq1$ and $-3$ has no square root in $\F_{q^3}$ as $q\equiv 2\pmod{3}$. This yields $x=a^{q}\notin \F_{q^3}$, which is a contradiction. Hence $\det D_1\neq 0.$ 

Now, let us suppose that $\det D_2=0$, we have 

\begin{equation*}
     x=\frac{\left(\frac{1}{y}+1\right)\pm\sqrt{-3}\left(\frac{1}{y}-1\right)}{2}.
\end{equation*}
By similar argument, we get $x=a \notin \mathbb F_{q^3}$, which is a contradiction. Hence, $\det D_2 \neq 0$. Thus, by Lemma \ref{L22}, $f_3(X)$ is a PP if and only if $a\neq 1$ and $\gcd(s, q-1)=1.$ This completes the proof.
\end{proof}

\begin{thm}\label{T34}
Let $q=p^m$, $p$ is an odd prime, $a\in \mu_{q^2+q+1} \subseteq \F_{q^3}$, and $f_4(X)=2X^{q^2}+(a+a^{2+q^2})X^q+(a^{1+q^2}+a^{1+2q^2})X+\Tr_{m}^{3m}(X)^s \in \F_{q^3}[X] $. Then
\begin{enumerate}
    \item [$(1)$] if $q\equiv 0\pmod{3}$, $f_4(X)$ is a PP if and only if $a$ is not a solution to the equations $X^{q+1}+X+1=0$ and $\gcd(s, q-1)=1$,
    \item [$(2)$]if $q\equiv 1\pmod{3}$, $f_4(X)$ is a PP if and only if $a$ is not a solution to any of the  equations $X^{q+1}-\frac{2}{1+\alpha}X^q+\frac{1-\alpha}{1+\alpha}=0$, $X^{q+1}-\frac{2}{1-\alpha}X^q+\frac{1+\alpha}{1-\alpha}=0$, $X^{q+1}-\frac{2}{1+\beta}X+\frac{1-\beta}{1+\beta}=0$ and $X^{q+1}-\frac{2}{1-\beta}X+\frac{1+\beta}{1-\beta}=0$, where $\alpha$, $\beta \in \F_{q^3}$ with $\alpha^2=-3=\beta^2$, and $\gcd(s, q-1)=1$,
      \item [$(3)$] if $q\equiv 2\pmod{3}$, $f_4(X)$ is a PP if and only if $a\neq 1$ and $\gcd(s, q-1)=1$.
\end{enumerate} 
   \end{thm}
\begin{proof}

We omit the proof of this theorem, as it can be proved using the same techniques as those in Theorem~\ref{T32} and Theorem \ref{T33}.
\end{proof}

In the next two theorems, we propose two classes of permutation polynomials over arbitrary finite fields.

\begin{thm}\label{T35}
Let $q=p^m$, where $p$ is an arbitrary prime, and $a\in \mu_{q^2+q+1} \subseteq \F_{q^3}$. Then $f_5(X)=2X^{q^2}+(a+a^2)X^q+(a^{1+q^2}+a^{2(1+q^2)})X+\Tr_{m}^{3m}(X)^s \in \F_{q^3}[X]$ is a PP if and only if $a\neq 1$ and $\gcd(s, q-1)=1$.
\end{thm}
\begin{proof}
We express $f_5(X)$ into the form specified in Lemma~\ref{L22}. Then, the determinants $\det D_1$ and $\det D_2$ corresponding to $f_5(X)$ are
    \begin{equation*}
    \det D_1=a^{3q+2q^2}-a^q+a^{2q}-a^{3q+q^2}+a^{q+q^2}-a^{2q+2q^2},
\end{equation*}
and 
\begin{equation*}
    \det D_2=a^{3+2q^2}-a+a^{2}-a^{3+q^2}+a^{1+q^2}-a^{2+2q^2}.
\end{equation*}
Assume that $f_5(X)$ is a PP over $\F_{q^3}$, then we have to show that $a\neq 1$ and $\gcd(s, q-1)=1$. If possible, suppose that $a=1$ then we get $\det D_1=0$ and $\det D_2=0$. Therefore, according to Lemma~\ref{L22}, $f_5(X)$ is not a PP over $\F_{q^3}$. Next, if
 $\gcd(s, q-1)\neq 1$, then again by Lemma~\ref{L22}, $f_5(X)$ is not a PP over $\F_{q^3}$. Hence, $a\neq 1$ and $\gcd(s, q-1)=1.$

 Conversely, suppose that $a\neq 1$ and $\gcd(s, q-1)=1$. In view of Lemma~\ref{L22}, it is enough to show that $\det D_1\neq0$ and $\det D_2\neq0$. Suppose on the contrary, that $\det D_1=0$, that is, $a^{3q+2q^2}-a^q+a^{2q}-a^{3q+q^2}+a^{q+q^2}-a^{2q+2q^2}=0$, which implies that $a^{2q+2q^2}-1+a^{q}-a^{2q+q^2}+a^{q^2}-a^{q+2q^2}=0.$ By using $x=a^q$ and $y=a^{q^2}$, we get $x^2y^2-1+x-x^2y+y-xy^2=0$, which is equivalent to $xy^2(x-1)+(x-1)-y(x-1)(x+1)=0.$ Since $x\neq 1$ as $a\neq 1$, so we have $xy^2+1-y(x+1)=0$, which implies that $(a^{q^2}-1)(a^{q+q^2}-1)=0.$ Therefore, either $a^{q^2}-1=0$ or $a^{q+q^2}-1=0$. But both of these cases imply that $a=1$, which is a contradiction. Hence $\det D_1\neq 0$.

Next, if possible, assume that $\det D_2=0$, that is, $a^{3+2q^2}-a+a^{2}-a^{3+q^2}+a^{1+q^2}-a^{2+2q^2}=0$, which gives $a^{2+2q^2}-1+a-a^{2+q^2}+a^{q^2}-a^{1+2q^2}=0.$ By substituting $x=a$ and $y=a^{q^2}$, we obtain $xy^2(x-1)+(x-1)-y(x^2-1)=0.$ As $x\neq 1$, so $xy^2+1-y(x+1)=0$, which implies that $(a^{q^2+1}-1)(a^{q^2}-1)=0.$ This would further imply that $a=1$, which is a contradiction. Hence $\det D_2\neq 0$. Therefore, by Lemma \ref{L22}, $f_5(X)$ is a PP over $\F_{q^3}$ if and only if $a\neq 1$ and $\gcd(s, q-1)=1.$
\end{proof}

\begin{thm}\label{T36}
Let $q=p^m$, where $p$ is an arbitrary prime, and $a\in \mu_{q^2+q+1} \subseteq \F_{q^3}$. Then $f_6(X)=2X^{q^2}+(a+a^{q+q^2})X^q+(a^{1+q^2}+a^q)X+\Tr_{m}^{3m}(X)^s \in \F_{q^3}[X]$ is a PP if and only if $a\neq 1$ and $\gcd(s, q-1)=1$.
\end{thm}
\begin{proof}
We  write $f_6(X)$ into the form described in Lemma~\ref{L22}. Then, the determinants $\det D_1$ and $\det D_2$ corresponding to $f_6(X)$ are given by $$\det D_1=a^{q+1}-a^q+a^{q^2+1}-a^{q^2}+a^{q^2+q}-a=\det D_2.$$

Suppose that $f_6(X)$ is a PP over $\F_{q^3}$, so we have to show that $a\neq 1$ and $\gcd(s, q-1)=1$. On the contrary, we suppose that $a=1$ then we get $\det D_1=0$ and $\det D_2=0$. Therefore, according to Lemma~\ref{L22}, $f_6(X)$ is not a PP over $\F_{q^3}$. Next, if possible, we assume that $\gcd(s, q-1)\neq 1$, then again by Lemma~\ref{L22}, $f_6(X)$ is not a PP over $\F_{q^3}$. Hence, $a\neq 1$ and $\gcd(s, q-1)=1.$

 Conversely, assume that $a\neq 1$ and $\gcd(s, q-1)=1$. In view of Lemma~\ref{L22}, it is enough to show that $\det D_1\neq0$ and $\det D_2\neq0$. If possible, let $\det D_1=0=\det D_2$, that is, $a^{q+1}-a^q+a^{q^2+1}-a^{q^2}+a^{q^2+q}-a=0$, which is equivalent to $  a^{q}-a^{2q+q^2}+a^{q^2}-a^{2q^2+q}+a^{2q^2+2q}-1=0$ as $a\in \mu_{q^2+q+1}$. By replacing $a^q$ with $x$ and $a^{q^2}$ with $y$, we get $x-x^2y+y-xy^2+x^2y^2-1=0$, which implies that $(1-xy)(x+y-xy-1)=0.$ So either $1-xy=0$ or $x+y-xy-1=0$. If $1-xy=0$ then we get $a=1$, which is a contradiction. If $x+y-xy-1=0$ then we have  $x=1$ or $y=1$ which gives $a=1$, so again we get a contradiction. Hence $\det D_1=\det D_2 \neq 0$. 
\end{proof}
\section{Compositional inverses and equivalence of permutation polynomials}\label{S4}
In this section, we provide the explicit expressions for the compositional inverses of the permutation polynomials obtained in Section \ref{S3}. Moreover, we study the inequivalence of our permutation polynomials with known classes of permutation polynomials of the same form, as well as among themselves.

 \subsection{Compositional inverses of permutation polynomials}\label{S41}
 
As alluded to in the introduction, Mullen~\cite{M} proposed the problem of computing the coefficients of the compositional inverses of permutation polynomials. In block ciphers, a permutation is often used as an S-box to build the confusion layer during the encryption process and the inverse
is needed while decrypting the cipher. In general, it is not easy to obtain the explicit expression for compositional inverses of PPs, except for several well-known classes of PPs such as linearized polynomials \cite{WL} and Dickson polynomials \cite{LQW}. The reader is referred to an excellent survey by Wang~\cite{W} on the various methods and their unification for computing compositional inverses of permutation polynomials. The problem of explicitly determining the compositional inverses of permutation polynomials is rather challenging. Here, we provide the coefficients of the compositional inverses of the permutation polynomials obtained in Section~\ref{S3}.

The following proposition provides the compositional inverse for the permutation polynomial of the form given in Lemma~\ref{L22} and can be gleaned from the proof of~\cite[Theorem 3]{WY1}.

\begin{prop}\label{L41}
   The compositional inverse of the PP $f(X)=A_1^{m_1}(X)+\displaystyle\sum_{i=2}^{n}{u_iA_i^{m_i}(X)}$ over $\mathbb F_{q^n}$ is given by 
   \begin{equation*}
      f^{-1}(X)=\theta_n \left(\lambda_1(X), \lambda_2(X),\ldots,\lambda_n(X)\right)^T,
   \end{equation*}
 where $\theta_n$ is the last row of $D_2^{-1}$, $\lambda_i(X)=a^{-qs_i}\left(\eta_i(X, X^q, \ldots, X^{q^{n-1}})^T\right)^{r_i}$  such that $\eta_i$ is the $i$th row of $D_1^{-1}$, $r_i$, $s_{i}$ are positive integers with the condition that $m_ir_i\equiv 1+s_i(q-1)\pmod {q^n-1}$, and the superscript T denotes the transpose of a vector.
\end{prop}

Using Proposition~\ref{L41}, we provide the explicit coefficient for the compositional inverses of the permutation polynomials obtained in Section~\ref{S3}. In fact, for each $1\leq t \leq 6$, the compositional inverse of the permutation polynomial $f_t$ is given by
\begin{equation*}
\begin{split}
f_{t}^{-1}(X)=b_{31}(a^{-qs_1}(c_{11}X+c_{12}X^q+c_{13}X^{q^2})^{r_1})+b_{32}(a^{-qs_2}(c_{21}X+c_{22}X^q+c_{23}X^{q^2})^{r_2})\\+b_{33}(a^{-qs_3}(c_{31}X+c_{32}X^q+c_{33}X^{q^2})^{r_3}),
\end{split}
\end{equation*}
where  $b_{ij}=\frac{d_{ij}}{\det(D_2)}$,  $c_{ij}=\frac{a_{ij}}{\det(D_1)}$, $1\leq i,j \leq3$, $m_1=m_2=1$ and $m_3=s.$ 

The values of $a_{ij}$'s and $d_{ij}$'s are given in Table \ref{T1}. 

\begin{table}[h]
\centering
\caption{The values of $a_{ij}$'s and $d_{ij}$'s required for the compositional inverses of PPs obtained in Section~\ref{S3}.}
\label{T1}
\begin{tabular}{|c| c| c| c |c| c| c|}
 \hline 
 t & 1  & 2 & 3 & 4 & 5  & 6\\ 
 \hline
 $a_{11}$ & $a^{2^kq}+a^{-2^k}$ & $a^{2q}-a^{-2}$ & $a^{-q}-a$ & $a^{q^2}-a^{-q}$ & $a^{q-1}-a^{q^2-1}$ & $a^{2q+1}-a^{q-1}$\\
 \hline
 $a_{12}$ & $1+a^{-2^k}$ & $a^{-2}-1$ & $a-1$ & $a^{-q}-1$ & $a^{q^2-1}-1$ & $a^{q-1}-1$\\
 \hline
  $a_{13}$ & $1+a^{2^kq}$ & $1-a^{2q}$ & $1-a^{-q}$ & $1-a^{q^2}$ & $1-a^{q-1}$ & $1-a^{2q+1}$\\
\hline
  $a_{21}$ & $a^{q}+a^{-1}$ & $a^{-1}-a^q$ & $a^{-1}-a^q$ & $a^{-1}-a^q$ & $a^{-1}-a^q$ & $a^{-1}-a^q$\\
\hline
  $a_{22}$ & $1+a^{-1}$ & $1-a^{-1}$ & $1-a^{-1}$ & $1-a^{-1}$ & $1-a^{-1}$ & $1-a^{-1}$\\
\hline
  $a_{23}$ & $1+a^{q}$ & $a^q-1$ & $a^q-1$ & $a^q-1$ & $a^{q}-1$ & $a^q-1$\\
\hline
  $a_{31}$ & $a^{q-2^k}+a^{2^k-1}$ & $a^{q-2}-a^{2q-1}$ & $a^{-q^2}-a^{q^2}$ & $1-a^{q^2-1}$ & $a^{-2}-a^{q-2}$ & $a^{2q-1}-a^{2q}$\\
\hline
 $a_{32}$ & $a^{(2^k+1)q}+a^{-1}$ & $a^{-1}-a^{-2}$ & $a^{-1}-a$ & $a^{-1}-a^{-q}$ & $a^{-1}-a^{q^2-1}$ & $a^{-1}-a^{q-1}$\\
\hline
 $a_{33}$ & $a^{2^kq}+a^q$ & $a^{2q}-a^q$ & $a^{-q}-a^q$ & $a^{q^2}-a^q $ & $a^{q-1}-a^q$ & $a^{2q+1}-a^q$\\
\hline
 $d_{31}$ & $1+a^{2^k}$ & $1-a^2$ & $1-a^{-1}$ & $1-a^q$ & $1-a^{2+q}$ & $1-a^{2+q^2}$ \\
\hline
 $d_{32}$ & $1+a$ & $a-1$ & $a-1$ & $a-1$ & $a-1$ & $a-1$\\
\hline
$d_{33}$ & $a+a^{2^k}$ & $a^2-a$ & $a^{-1}-a$ & $a^q-a$ & $a^{2+q}-a$ & $a^{2+q^2}-a$\\
 \hline
\end{tabular}
\end{table}

\subsection{Inequivalence of permutation polynomials}\label{S42}
First, we establish the inequivalence among the permutation polynomials obtained in Section~\ref{S3}, as presented in Table~\ref{T2}. Subsequently, we investigate the inequivalence of the permutation polynomials given in Table~\ref{T2} with known classes of permutation polynomials of the same form as listed in Table~\ref{T3}. Here, we consider both linear and quasi-multiplicative equivalence of PPs, which are defined below.
\begin{def1} \textup{\cite{LC}} Two polynomials $f(X)$ and $g(X)$ in $\F_{q}[X]$ are linearly equivalent if there exist two linearized permutation polynomials $A_1(X)$ and  $A_2(X)$ over $\F_{q}$ such that 
\begin{equation*}
    f(X)=A_2(g(A_1(X))).
\end{equation*}
\end{def1} 
\begin{def1}\textup{\cite{WYDM}} Two permutation polynomials $f(X)$ and $g(X)$ in $\F_{q}[X]$ are quasi-multiplicative $(QM)$ equivalent if there exist an integer $1\leq d \leq q-2$ such that $\gcd(d, q-1)=1$ and $f(X)=ag(bX^d)$ for some $a,b \in \F_q^*.$
\end{def1}

\subsubsection{Inequivalence of the PPs listed in Table~\ref{T2} among themselves}\label{S421}

In Table~\ref{T2}, the PPs $f_{t}$'s for $1 \leq t \leq 6$ are linearly inequivalent as $L_{i}'(X)=L_{j}'(X)$  but $L_{i}(X) \neq L_{j}(X)$ for $1 \leq i \neq j \leq 6$.

Now, we discuss the QM-equivalence of the PPs given in Table~\ref{T2}. 
\begin{prop}\label{P1}
 The PPs of Table~\ref{T2} are QM-inequivalent among themselves.
\begin{proof}
If possible, we suppose that there exists an integer d satisfying $1 \leq d \leq q^3-2$ and  $\gcd(d, q^3-1)=1$ such that for some  $a,b \in \F_{q^3}^*$ and distinct $i,j$, $1 \leq i \neq j \leq 6$, we have $f_{i}(X)=af_{j}(bX^d)$. Now, we consider two cases. First, we suppose that $d$ is of the form $p^k$ for some integer $k\geq 0$, where $p$ is the characteristic of the field $\F_{q^3}$. This implies that $f_{i}(X)=A_{2}(f_{j}(A_{1}(X)))$, where $A_{2}(X)=aX$ and $A_{1}(X)=bX^{p^k}$, which are linearized permutation polynomials over $\F_{q^3}$. This contradicts the linear inequivalence of $f_{t}$'s. Now, consider the other case when $d \neq p^k$ for any integer $k\geq 0$. Then, it is clear that $L_{i}(X)+\Tr_{m}^{3m}(X)^s \neq aL_{j}(bX^d)+a\Tr_{m}^{3m}(bX^d)^s$, $i \neq j$, as $L_j(X^d)$ is never a linearized polynomial in this case. Hence $f_{t}$'s are QM-inequivalent to each other.
\end{proof}
\end{prop}

\begin{table}[h]
\centering
\caption{PPs of the form $f_{t}(X)=L_{t}(X)+L_{t}'(X)^s, ~1\leq t \leq 6$.}
\label{T2}
\begin{tabular}{c c c c c}
 \hline 
 t & $L_{t}(X)$ & $L_{t}'(X)$ & Condition on $s$ & References\\ 
 \hline
 1 & $(a+a^{2^k})X^q+(a^{1+q^2}+a^{2^{k}(1+q^2)})X$ & $\Tr_{m}^{3m}(X)$& $\gcd(s, q-1)=1$ & Theorem \ref{T31}\\
 \hline
 2 & $2X^{q^2}+(a+a^q)X^q+(a^{1+q^2}+a^{1+q})X$ & $\Tr_{m}^{3m}(X)$ & $\gcd(s, q-1)=1$ & Theorem \ref{T32}\\
 \hline
 3 & $2X^{q^2}+(a+a^{2+q})X^q+(a^{1+q^2}+a^{2+q^2})X$ & $\Tr_{m}^{3m}(X)$ & $\gcd(s, q-1)=1$ & Theorem \ref{T33}\\
 \hline
 4 & $2X^{q^2}+(a+a^{2+q^2})X^q+(a^{1+q^2}+a^{1+2q^2})X$ & $\Tr_{m}^{3m}(X)$ & $\gcd(s, q-1)=1$ & Theorem \ref{T34}\\
 \hline
 5 & $2X^{q^2}+(a+a^2)X^q+(a^{1+q^2}+a^{2(1+q^2)})X$ & $\Tr_{m}^{3m}(X)$&  $\gcd(s, q-1)=1$ & Theorem \ref{T35}\\
 \hline
 6 & $2X^{q^2}+(a+a^{q+q^2})X^q+(a^{1+q^2}+a^q)X$ & $\Tr_{m}^{3m}(X)$ &  $\gcd(s, q-1)=1$ & Theorem \ref{T36}\\
 \hline
\end{tabular}
\end{table} 

\subsubsection{Inequivalence of the PPs listed in Table~\ref{T2} with permutation binomials.}\label{S422}
We discuss the linear inequivalence of the PPs listed in Table~\ref{T2} with the permutation polynomials of the form $bX^s+aX^r$, where $r$ and $s$ are positive integers such that $\gcd(s, q-1)=1$. It is important to note that if $r\neq p^i$, where $i$ is a positive integer and $p$ is the characteristic of the field $\F_{q^3}$, then for any linearized permutation polynomials $A_{2}(X)$ and $A_{1}(X)$, the composition of maps $A_2((bX^s+aX^r) (A_1(X)))$ can never be of the form $L_{t}(X)+L_{t}'(X)^s$ as $A_{2}(a{A_{1}(X)}^r)$ is not a linearized polynomial. Thus, in the following proposition, we shall assume that $r=p^i$, where $i$ is a positive integer and $p$ is the characteristic of $\F_{q^3}.$

\begin{prop}\label{P2} The PPs listed in Table~\ref{T2} are linearly inequivalent to any permutation binomial of the form $X^s+aX^r$ where $r=p^i$, $p$ is the characteristic of $\F_{q^3}$, $i, r$ and $s$ are positive integers such that $\gcd(s,q-1)=1$.
\begin{proof} Let $f(X)=X^s+aX^r$ be a permutation binomial and $f_{t}(X)=L_{t}(X)+L_{t}'(X)^s$ be a PP of the form given in Table~\ref{T2} over $\mathbb F_{q^3}$. If possible, suppose that there exist two linearized permutation polynomials $A_{1}(X)$ and $A_{2}(X)$ over $\F_{q^3}$ such that $f_{t}(X)=A_{2}(f(A_{1}(X)))=A_{2}(A_{1}(X)^s)+A_{2}(aA_{1}(X)^r).$ This forces that $A_{2}(X)=X^{p^k}$ for some $k\geq 0$. Thus, it follows that $L_{t}(X)=(aA_{1}(X)^r)^{p^k}$ and $L_{t}'(X)=(A_{1}(X))^{p^k}.$ Since $L_{t}'(X)=\Tr_{m}^{3m}(X)$, we have $\Tr_{m}^{3m}(X)=(A_{1}(X))^{p^k}$. Now, $A_1(X)$ being a PP implies that $\Tr_{m}^{3m}(X)$ is a PP, which is certainly not true. Hence, PPs in Table~\ref{T2} are linearly inequivalent to any permutation binomial. 
\end{proof} 
\end{prop}
\subsubsection{Inequivalence of the PPs listed in Table~\ref{T2} with known classes of PPs of the form $L_{i}(X)+L_{i}'(X)^s$}\label{S423}

In Subsections $4.2.1$ and $4.2.2$, it is shown that the PPs in Table~\ref{T2} are QM and linearly inequivalent to each other and linearly inequivalent to the permutation binomials, respectively. Therefore, it remains to show that the PPs constructed in Section \ref{S3} (and listed in Table~\ref{T2}) are QM and linearly inequivalent to the known classes of PPs of the form $L_{i}(X)+L_{i}'(X)^s$ over $\mathbb F_{q^3}$ to establish the novelty of our results. To the best of our knowledge, we list all the PPs of the form $L_{i}(X)+L_{i}'(X)^s$ over $\mathbb F_{q^3}$ in Table~\ref{T3} and compare them with the ones given in Table~\ref{T2}. Note that whenever $\delta$ occurs in any PP given in Table~\ref{T3}, we will set $\delta=0$ for comparison with the PPs listed in Table~\ref{T2}.

\begin{longtable}{|m{1em}| m{7em}| m{7em}| m{12.5em}| m{10em}|}

\caption{Known classes of PPs of the form $L_{i}(X)+L_{i}'(X)^s$ over $\F_{q^3}$}\label{T3}
\endlastfoot
\hline 
  i & $ L_{i}(X)$ &  $L_{i}'(X)$ & Conditions on $s, m, n$ and $\delta$ & 
      References\\ 
     
 \hline
  1 & $X^{p^k}+X$ & $X^{p^k}-X+\delta$ & $s=\frac{p^n+1}{2}$, $\delta \in \F_{p^n}$  & \cite[Theorem 2]{LHT}\\
 \hline
  2 & $X$ & $X^{p^k}-X+\delta$ & $s=\frac{p^n-1}{2}+p^k$, $\delta \in \F_{p^n}$  & \cite[Theorem 4]{LHT}\\
 \hline
 3 & $X^{3^k}+X$ & $X^{3^k}-X+\delta$ & $s=\frac{3^n-1}{2}+3^k$,  & \cite[Theorem 5]{LHT}\\
 
   &  &  &  $\Tr_{m}^{3m}(\delta)=0$, $\delta \in \F_{3^n}$ &\\
 \hline
  4 & $X^{p^k}+X$ & $X^{p^k}-X+\delta$ & $s=\frac{p^n-1}{2}+p^{2k}$, $\delta \in \F_{p^n}$  & \cite[Theorem 6]{LHT}\\
 \hline
  5 & $-X^{p^k}-X$ & $X^{p^k}-X+\delta$ & $s=\frac{p^n-1}{2}+p^{2k}$, $\delta \in \F_{p^n}$  & \cite[Theorem 7]{LHT}\\
 \hline
  6 & $\phi(X)$ & $\Tr(X)$ &   & \cite[Corollary 1.6]{TW}\\
 \hline
  7 & $\gamma^2X+(\gamma^2+\gamma)X^{2^m}$   & $(1+\gamma)X+\gamma X^{2^m}$ & $m\equiv 1 \pmod 3$,& \cite[Theorem 3.1]{VS}\\
  & $+\gamma X^{2^{2m}}$  & $+(1+\gamma)X^{2^{2m}}$ &  $\gcd(s, q-1)=1$&\\
 \hline
 8 & $\gamma^2X+(\gamma^2+\gamma)X^{2^m}$  & $(\gamma+\gamma^2) X^{2^m}+\gamma X^{2^{2m}}$ & $m\equiv 1 \pmod 3$, & \cite[Theorem 3.1]{VS}\\
  
  & $+\gamma X^{2^{2m}}$  &  & $\gcd(s, q-1)=1$&\\
 \hline 
  9 & $\gamma^2X+\gamma X^{2^m}$  & $(1+\gamma)X+(1+\gamma) X^{2^m}$ & $m\equiv 2 \pmod 3$, & \cite[Theorem 3.2]{VS}\\
 
   & $+(\gamma+\gamma^2 )X^{2^{2m}}$  & $+\gamma X^{2^{2m}}$ &  $\gcd(s, q-1)=1$&\\
 \hline 
 10 & $\gamma^2X+\gamma X^{2^m}$  & $\gamma X^{2^m}+(\gamma+\gamma^2)X^{2^{2m}}$ & $m\equiv 2 \pmod 3$, & \cite[Theorem 3.2]{VS}\\

   & $+(\gamma+\gamma^2 )X^{2^{2m}}$  & &  $\gcd(s, q-1)=1$&\\
 \hline 
 11 & $aX^{q^2}+a\gamma X^q$,  & $X^{q^2}+\gamma X^q$  &  & \cite[Theorem 2]{WY1}\\
 & $+cX$, $a,c \in \F_{q^3}$ & $+\gamma^{1+q^2}X$, $\gamma \in \F_{q^3}$ & &\\
 \hline
 12 &  $bX$,  $b\in \F_{2^m}\setminus \F_{2}$ & $X^{2^m}+X+\delta$ & ${\small s (2^m+1)\equiv 1 \pmod{2^n-1}}$, &  \cite[Proposition 1]{XFZ}\\
 & & & $\delta \in \F_{2^n}$ & \\
   \hline
 13 & $bX$, $b\in \F_{2^m}\setminus \F_{2}$   & $X^{2^m}+X+\delta$ & $s=2^{2m}+1$, $\delta \in \F_{2^{3m}}$ & \cite[Proposition 2]{XFZ}\\
 \hline
 14 &  $bX$, $b\in \F_{p^m}\setminus \{0, 1\}$ & $X^{p^m}-X+\delta$ & $(p^m-1)s \equiv 0\pmod{p^n-1}$, & \cite[Proposition 5]{XFZ}\\
   & & & $\delta \in \F_{p^n}$ & \\
   \hline
15 & $L(X)$ & $B(X)$ &   & \cite[Theorem 3.1]{YD2}\\
  \hline
  16 & $L(X)$ & $L'(X)$ & $(q^k-1)s\equiv 0\pmod{q^n-1}$  & \cite[Corollary 6.2]{YD2}\\
   &  &   &  $\gcd(n, k)>1$ &\\
 \hline
 17 & $X$ & $X^{q^k}-X+\delta$ & $(q^k-1)s\equiv 0\pmod{q^n-1}$  & \cite[Corollary 6.3]{YD2}\\
   &  &   &  $\gcd(n, k)>1$ &\\
 \hline
  18 & $X$ & $X^{2^k}+X+\delta$  & $(2^k-1)s\equiv 0\pmod{2^n-1}$, & \cite[Proposition 2]{ZZH}\\
 
  &  &   &  $\gcd(n, k)>1$, $\delta \in \F_{2^n}$  &\\
 \hline

\end{longtable}

\begin{prop}\label{P3}
The PPs listed in Table~\ref{T2} are both linearly and QM-inequivalent to the PPs given in the first five rows of Table \ref{T3}.
\end{prop}
\begin{proof}
Notice that all the PPs given in the first five rows of Table \ref{T3}, which are from \cite{LHT}, are over the finite field $\F_{p^n}$, where $n$ is a positive integer. Thus, to compare them with ones in Table~\ref{T2}, we may take $n=3m$ so as to get $p^n=q^3$, where $q=p^m$. It is important to note that if $s$ is such that $\gcd(s,q-1)\neq 1$ in all the PPs from the first to the fifth rows in Table~\ref{T3}, then there is no point in comparing them with the PPs listed in Table~\ref{T2}, as the conditions on $s$ are different. However, if $\gcd(s,q-1)= 1$, we handle the first row separately from the second to fifth rows of Table~\ref{T3}.

\textbf{Case 1.} We first discuss the equivalence of the PPs given in Table~\ref{T2} with the first row of Table~\ref{T3}. We can see that the exponent $s$ in the PPs of first row of Table~\ref{T3} is fixed. However, the exponents $s$ in the PPs of Table~\ref{T2} include all integers that satisfy the condition $\gcd(s,q-1)= 1$ and are not fixed, making the PPs in Table~\ref{T2} more general.

\textbf{Case 2.} Here, we consider the PPs from the second to the fifth rows of Table~\ref{T3}, where the exponent $s$ is of the form $s=\frac{p^n-1}{2}+\ell$, with $\ell \geq 1$ and $\gcd(s,q-1)= 1$, and compare them with the PPs given in Table~\ref{T2}. Since $\gcd(q-2,q-1)=1$, so the exponent $q-2$ is included in all the PPs given in Table~\ref{T2}. However, it is evident that $q-2$ can never be expressed in the form $s=\frac{p^n-1}{2}+\ell$ for any positive integer $\ell$. Therefore, the exponent $q-2$ is not included in the PPs from the second to the fifth row of Table~\ref{T3}. Thus, the PPs given in Table~\ref{T2} are more general.
\end{proof}

 \begin{prop}\label{P4}
 The PPs in Table~\ref{T2} are both linearly and QM-inequivalent to the PPs in the sixth row of Table~\ref{T3}.
\end{prop}
\begin{proof}
We can observe that the PPs in Table~\ref{T2} are linearly inequivalent to the PPs in the sixth row of Table \ref{T3} because $L_{t}'(X)=\Tr(X)$ and $\L_{t}(X)\neq \phi(X)$ for $1\leq t \leq 6$ (as $L_{t}(X)\in \F_{q^3}[X]$ and $\phi(X)\in \F_{q}[X])$. Furthermore, using a similar argument as in the proof of Proposition \ref{P1}, we can conclude that the PPs in Table~\ref{T2} are QM-inequivalent to the PPs in the sixth row of Table \ref{T3}.
\end{proof}

\begin{prop}\label{P5}
 The PPs over fields of even characteristic in Table~\ref{T2} are both linearly and QM-inequivalent to the PPs from the seventh to the tenth rows of Table~\ref{T3}. 
\end{prop}
\begin{proof}
The PPs over fields of even characteristic in Table~\ref{T2}  are more general than those in the rows from seventh to tenth of Table~\ref{T3} because the PPs in Table~\ref{T2}  are over the finite field $\F_{q^3}$, where $q=2^m$ and  $m$ is an arbitrary positive integer, unlike the PPs in Table \ref{T3} for which $m\equiv 1,2 \mod 3$.     
\end{proof}

\begin{prop}\label{P6}
    The PPs in Table~\ref{T2} are QM-inequivalent to the PPs in the eleventh row of Table \ref{T3}.
\end{prop}
\begin{proof}
Suppose that $g(X)=g_{1}(X)+g_{2}(X)^s$ is  a PP in the eleventh row of Table~\ref{T3}, where $g_1(X)=aX^{q^2}+a\gamma X^q+cX$ and $g_{2}(X)=X^{q^2}+\gamma X^q+\gamma^{1+q^2}X.$ If possible, assume that $f_{t}(X)=\beta g(\alpha X^d)$ for some $\alpha, \beta \in \F_{q^3}^*$,  where $\gcd(d, q^3-1)=1$ and for $1\leq t \leq 6$, $f_t(X)$ is a PP from Table~\ref{T2}. By comparing the exponents on both the sides, we have $d=1$. Further, by comparing the coefficients of $\beta g_2(\alpha X)$ and $\Tr(X)$, we have $\beta \alpha^{q^2}=1$, $\beta \alpha ^q \gamma=1$, and $\beta \alpha \gamma^{1+q^2}=1$, implying that the coefficients of $X^{q^2}$ and $X^q$ in  $\beta g_{1}(\alpha X)$ are same. Consequently, the coefficients of $X^{q^2}$ and $X^q$ in $L_t(X)$, for $1\leq t \leq 6$, are same, which is obviously a contradiction. Therefore, $f_{t}(X)\neq \beta g(\alpha X)$ for any $\alpha, \beta \in \F_{q^3}^*.$
\end{proof}

\begin{prop}\label{P7}
 The PPs listed in Table~\ref{T2} are both linearly and QM-inequivalent to the PPs from the twelfth to fourteenth rows of Table \ref{T3}.
    \end{prop}
\begin{proof}
We discuss the inequivalence of PPs given in Table~\ref{T2} with the PPs from the twelfth to fourteenth rows of Table \ref{T3}. Since $L(X)=bX$ in all the PPs from the twelfth to fourteenth rows of Table \ref{T3}. Now, if possible, assume that the PP $f_{t}(X)$ of Table~\ref{T2} is linearly equivalent to the PPs given in the PPs from the twelfth to fourteenth rows of Table \ref{T3}. It follows that $L_{t}(X)=A_2(bA_{1}(X))$ for some linearized permutation polynomials $A_{1}(X)$ and $A_{2}(X)$, which implies that $L_{t}(X)$ (in Table~\ref{T2}) is a PP, a contradiction. This completes the proof.
 \end{proof}

\begin{prop}\label{P8}
The PPs in Table~\ref{T2} are both linearly and QM-inequivalent to the PPs in the last three rows of Table~\ref{T3}.
\end{prop}
\begin{proof}
First, we will discuss the inequivalence of PPs given in Table~\ref{T2} with the PP $f(X)=L(X)+L'(X)^s$ in the sixteenth row of Table~\ref{T3} when $n=3$. Note that, when $n=3$, the condition $d=\gcd(n,k)> 1$ of \cite[Corollary 6.2]{YD2} implies that $d=3$. Therefore, $L'(X)$ becomes $0$, and so $f(X)=L(X)$. Thus, it follows that $f(X)$ is both linearly and  QM-inequivalent to $f_t(X)$, for $1\leq t\leq 6$. 

Next, the inequivalence of PPs given in Table~\ref{T2} with the last two rows of Table~\ref{T3} readily follows from the proof of the Proposition~\ref{P7}. 
\end{proof}

\begin{rmk}
It is easy to see that our classes of permutation polynomials are not subclasses of the PPs given in the fifteenth row of Table~\ref{T3} since $L(X)$ and $B(X)$ are in $\F_{q}[X]$.
\end{rmk}

\section{Conclusion}\label{S5}
In this paper, we proposed six new classes of permutation polynomials over $\F_{q^3}$ and provided the explicit expressions for their compositional inverses. Here, we used a technique given by Wu and Yuan (2023) and analyzed certain equations over finite fields. We also show that these classes of permutation polynomials are not equivalent to the known classes of permutation polynomials of the same type. It would be interesting to explore such permutation polynomials over $\F_{q^3}$ when $\delta \neq 0.$

\end{document}